\documentclass[eqnumis, pdf]{hjms}
\usepackage{amsfonts}
\usepackage{amsmath}
\usepackage{footnote}
\usepackage{multicol}
\usepackage{booktabs}
\usepackage[flushleft]{threeparttable}
\usepackage[font=small,labelfont=bf]{caption}

\DeclareMathOperator{\fc}{\xrightarrow[]{mo}}
\DeclareMathOperator{\uoc}{\xrightarrow[]{uo}}
\DeclareMathOperator{\oc}{\xrightarrow[]{o}}

\begin{document}
%please do not change this part%%%%%
%\hinfo{XX}{x}{2019}{1}{\lpage}{10.15672/HJMS.xx}
%%%%%%%%%%%%%%%%%%%%%%%%%%%%%%
%\giris
%{ } %yazarlar
%{ } %başlık
%{ } %ilk sayfa no
%author, title, first page number
%%%%%%%%%%%%%%%%%%%%%%%%%%%%%%
%Corresponding Author Email:
% 
%%%%%%%%%%%%%%%%%%%%%%%%%%%%%%

\markboth{A. Ayd{\i}n}{$mo$-convergence}

\title{Multiplicative order convergence in $f$-algebras}

\author{Abdullah Ayd{\i}n\coraut$^{1}$}

\address{$^1$Department of Mathematics, Mu\c{s} Alparslan University, Mu\c{s}, Turkey\\}
\emails{aaydin.aabdullah@gmail.com}
\maketitle
\begin{abstract}
A net $(x_\alpha)$ in an $f$-algebra $E$ is said to be multiplicative order convergent to $x\in E$ if $\left|x_\alpha-x\right|u\oc 0$ for all $u\in E_+$. In this paper, we introduce the notions $mo$-convergence, $mo$-Cauchy, $mo$-complete, $mo$-continuous and $mo$-KB-space. Moreover, we study the basic properties of these notions.
\end{abstract}% the abstract
\subjclass{46A40, 46E30}  % AMS subject classifications
\keywords{$mo$-convergence, $f$-algebra, $mo$-KB-space, vector lattice}        % Keywords

%please do not change this part%%%%%%%%%%
\hinfoo{DD.MM.YYYY}{DD.MM.YYYY} %receive, accept
%%%%%%%%%%%%%%%%%%%%%%%%%%%%%%%%%%%

%Accepted Date: 
%Received Date: 
%DOI: 

\section{Introductory facts}\label{sec1}
In spite of the nature of the classical theory of Riesz algebra and $f$-algebra, as far as we know, the concept of convergence in $f$-algebras related to multiplication has not been done before. However, there are some close studies under the name unbounded convergence in some kinds of vector lattices; see for example \cite{AAyd,AEEM2,AGG,GTX}. Our aim is to introduce the concept of $mo$-convergence by using the multiplication in $f$-algebras.

First of all, let us remember some notations and terminologies used in this paper. Let $E$ be a real vector space. Then $E$ is called \textit{ordered vector space} if it has an order relation $\leq$ (i.e, $\leq$ is reflexive, antisymmetric and transitive) that is compatible with the algebraic structure of $E$ that means $y\leq x$ implies $y+z\leq x+z$ for all $z\in E$ and $\lambda y\leq \lambda x$ for each $\lambda\geq 0$. An ordered vector $E$ is said to be {\em vector lattice} (or, {\em Riesz space}) if, for each pair of vectors $x,y\in E$, the supremum $x\vee y=\sup\{x,y\}$ and the infimum $x\wedge y=\inf\{x,y\}$ both exist in E. Moreover, $x^+:=x\vee 0$, $x^-:=(-x)\vee0$, and $\lvert x\rvert:=x\vee(-x)$ are called the {\em positive} part, the {\em negative} part, and the {\em absolute value} of $x\in E$, respectively. Also, two vector $x$, $y$ in a vector lattice is said to be {\em disjoint} whenever $\lvert x\rvert\wedge\lvert y\rvert=0$. A vector lattice $E$ is called \textit{order complete} if $0\leq x_\alpha\uparrow\leq x$ implies the existence of $\sup{x_\alpha}$ in $E$. A subset $A$ of a vector lattice is called \textit{solid} whenever $\left|x\right|\leq\left|y\right|$ and $y\in A$ imply $x\in A$. A solid vector subspace is referred to as an \textit{order ideal}. An order closed ideal is referred to as a \textit{band}. A sublattice $Y$ of a vector lattice is majorizing $E$ if, for every $x\in E$, there exists $y\in Y$ with $x\leq y$. A partially ordered set $I$ is called {\em directed} if, for each $a_1,a_2\in I$, there is another $a\in I$ such that $a\geq a_1$ and $a\geq a_2$ (or, $a\leq a_1$ and $a\leq a_2$). A function from a directed set $I$ into a set $E$ is called a {\em net} in $E$. A net $(x_\alpha)_{\alpha\in A}$ in a vector lattice $X$ is called \textit{order convergent} (or shortly, \textit{$o$-convergent}) to $x\in X$, if there exists another net $(y_\beta)_{\beta\in B}$ satisfying $y_\beta \downarrow 0$, and for any $\beta\in B$ there exists $\alpha_\beta\in A$ such that $|x_\alpha-x|\leq y_\beta$ for all $\alpha\geq\alpha_\beta$. In this case, we write $x_\alpha\xrightarrow{o} x$; for more details see for example \cite{ABPO,Vul,Za}.

A vector lattice $E$ under an associative multiplication is said to be a \textit{Riesz algebra} whenever the multiplication makes $E$ an algebra (with the usual properties), and in addition, it satisfies the following property: $x,y\in E_+$ implies $xy\in E_+$. A Riesz algebra $E$ is called \textit{commutative} if $xy=yx$ for all $x,y\in E$. A Riesz algebra $E$ is called \textit{$f$-algebra} if $E$ has additionally property that $x\wedge y=0$ implies $(xz)\wedge y=(zx)\wedge y=0$ for all $z\in E_+$; see for example \cite{ABPO}. A vector lattice $E$ is called \textit{Archimedean} whenever $\frac{1}{n}x\downarrow 0$ holds in $E$ for each $x\in E_+$. Every Archimedean $f$-algebra is commutative; see Theorem 140.10 \cite{Za}. Assume $E$ is an Archimedean $f$-algebra with a multiplicative unit vector $e$. Then, by applying Theorem 142.1(v) \cite{Za}, in view of $e=ee=e^2\geq0$, it can be seen that $e$ is a positive vector. In this article, unless otherwise, all vector lattices are assumed to be real and Archimedean, and so $f$-algebras are commutative.

Recall that a net $(x_\alpha)$ in a vector lattice $E$ is \textit{unbounded order convergent} (or shortly, \textit{$uo$-convergent}) to $x\in E$ if $|x_\alpha-x|\wedge u\xrightarrow{o}0$ for every $u\in E_+$. In this case, we write $x_\alpha\xrightarrow{uo}x$; see for example \cite{GTX} and \cite{AAyd,AEEM2,AGG}. Motivated from this definition, we give the following notion.
\begin{definition}
Let $E$ be an $f$-algebra. A net $(x_\alpha)$ in $E$ is said to be {\em multiplicative order convergent} to $x\in E$ (shortly, $(x_\alpha)$ $mo$-converges to $x$) if $\left|x_\alpha-x\right|u\oc 0$ for all $u\in E_+$. Abbreviated as $x_\alpha\fc x$. 
\end{definition}

It is clear that $x_\alpha\fc x$ in an $f$-algebra $E$ implies $x_\alpha y\fc xy$ for all $y\in E$ because of $\left|xy\right|=\left|x\right|\left|y\right|$ for all $x,y\in E$. We shall keep in mind the following useful lemma, obtained from the property of  $xy\in E_+$ for every $x,y\in E_+$.
\begin{lemma}\label{basic monotonicity}
If $y\leq x$ is provided in an $f$-algebra $E$ then $uy\leq ux$ for all $u\in E_+$.
\end{lemma}

Recall that multiplication by a positive element in $f$-algebras is a vector lattice homomorphism, i.e., $u(x\wedge y)=(ux)\wedge(uy)$ and $u(x\vee y)=(ux)\vee(uy)$ for every positive element $u$; see for example Theorem 142.1(i) \cite{Za}. We will denote an $f$-algebra $E$ as \textit{infinite distributive} $f$-algebra whenever the following condition holds: if $\inf(A)$ exists for any subset $A$ of $E_+$ then the infimum of the subset $uA$ exists and $\inf(uA)=u\inf(A)$ for each positive vector $u\in E_+$. For a net $(x_\alpha)\downarrow 0$ in an infinite distributive $f$-algebra, the net $(ux_\alpha)$ is also decreasing to zero for all positive vector $u$.

\begin{remark}\label{order con iff f-conv}
The order convergence implies the $mo$-convergence in infinite distributive $f$-algebras. The converse holds true in $f$-algebras with multiplication unit. Indeed, assume a net $(x_\alpha)_{\alpha \in A}$ order converges to $x$ in an infinite distributive $f$-algebra $E$. Then there exists another net $(y_\beta)_{\beta\in B}$ satisfying $y_\beta \downarrow 0$, and, for any $\beta\in B$, there exists $\alpha_\beta\in A$ such that $|x_\alpha-x|\leq y_\beta$. Hence, we have $|x_\alpha-x|u\leq y_\beta u$ for all $\alpha\geq\alpha_\beta$ and for each $u\in E_+$. Since $y_\beta \downarrow$, we have $uy_\beta\downarrow$ for each $u\in E_+$ by Lemma \ref{basic monotonicity}, and $\inf(uy_\beta)=u\inf(y_\beta)=0$ because of $\inf(y_\beta)=0$. Therefore, $\left|x_\alpha-x\right|u\oc 0$ for each $u\in E_+$. That means $x_\alpha\fc x$.
	
For the converse, assume $E$ is an $f$-algebra with multiplication unit $e$ and $x_\alpha\fc x$ in $E$. That is, $\left|x_\alpha-x\right|u\oc 0$ for all $u\in E_+$. Since $e\in E_+$, in particular, choose $u=e$, and so we have $\left|x_\alpha-x\right|=\left|x_\alpha-x\right|e\oc 0$, or $x_\alpha\oc x$ in $E$.
\end{remark}

By considering  Example 141.5 \cite{Za}, we give the following example.
\begin{example}
Let $[a,b]$ be a closed interval in $\mathbb{R}$ and let $E$ be vector lattice of all reel continuous functions on $[a,b]$ such that the graph of functions consists of a finite number of line segments. In view of Theorem 141.1 \cite{Za}, every positive orthomorphism $\pi$ in $E$ is trivial orthomorphism, i.e., there is a reel number $\lambda$ such that $\pi (f)=\lambda f$ for all $f\in E$. Therefore, a net of positive orthomorphism $(\pi_\alpha)$ is order convergent to $\pi$  iff it is $mo$-convergent to $\pi$ whenever the multiplication is the natural multiplicative, i.e., $\pi_1\pi_2(f)=\pi_1(\pi_2f)$ for all $\pi_1,\pi_2\in Orth(E)$ and all $f\in E$. Indeed, $Orth(E)$ is Archimedean $f$-algebra with the identity operator as a unit element; see Theorem 140.4 \cite{Za}. So, by applying Remark \ref{order con iff f-conv}, the $mo$-convergence implies the order convergence of the net $(\pi_\alpha)$. 

Conversely, assume the net of positive orthomorphisms $\pi_\alpha\oc \pi$ in $Orth(E)$. Then we have  $\pi_\alpha(f)\oc \pi(f)$ for all $f\in E$; see Theorem VIII.2.3 \cite{Vul}. For fixed $0\leq\mu\in Orth(E)$, there is a reel number $\lambda_\mu$ such that $\mu(f)=\lambda_\mu f$ for all $f\in E$. Since  $\left|\pi_\alpha(f)-\pi(f)\right|=\left|\lambda_{\pi_\alpha}f-\lambda_\pi f\right|\oc 0$, we have
$$
\left|(\pi_\alpha)f-(\pi) f\right|\mu=\left|\mu\lambda_{\pi_\alpha}f-\mu\lambda_\pi f\right|=\left|\lambda_\mu\lambda_{\pi_\alpha}f-\lambda_\mu\lambda_\pi f\right|=\left|\lambda_\mu\right|\left|\lambda_{\pi_\alpha}f-\lambda_\pi f\right|\oc 0
$$
for all $f\in E$. Since $\mu$ is arbitrary, we get $\pi_\alpha\fc \pi$.
\end{example}

\section{Main results}
We begin the section with the next list of properties of the $mo$-convergence which follows directly from Lemma \ref{basic monotonicity}, and the inequalities $\left|x-y\right| \leq \left|x-x_\alpha\right|+\left|x_\alpha-y\right|$ and $\left| \left|x_\alpha\right|-\left|x\right| \right|\leq\left| x_\alpha-x\right|$.
\begin{lemma}
Let $x_\alpha\fc x$ and $y_\alpha\fc y$ in an $f$-algebra $E$. Then the following holds:
\begin{enumerate}
\item[(i)] $x_\alpha\fc x$ iff $(x_\alpha- x)\fc 0$;
\item[(ii)] if $x_\alpha\fc x$ then $y_\beta\fc x$ for each subnet $(y_\beta)$ of $(x_\alpha)$;
\item[(iii)] suppose $x_\alpha\fc x$ and $y_\beta\fc y$ then $ax_\alpha+by_\beta\fc ax+by$ for any $a,b\in \mathbb{R}$;
\item[(iv)] if $x_\alpha \fc x$ and $x_\alpha \fc y$ then $x=y$; 
\item[(v)] if $x_\alpha \fc x$ then $\lvert x_\alpha\rvert \fc \lvert x \rvert$.
\end{enumerate}
\end{lemma}

Recall that an order complete vector lattice $E^\delta$ is said to be
an order completion of the vector lattice $E$ whenever $E$ is Riesz isomorphic to a majorizing order dense vector lattice subspace of $E^\delta$. Every Archimedean Riesz space has a unique order completion; see Theorem 2.24 \cite{ABPO}.
\begin{proposition}
Let $(x_\alpha)$ be a net in an $f$-algebra $E$. Then $x_\alpha\fc 0$ in $E$ iff $x_\alpha\fc 0$ in the order completion $E^\delta$ of $E$.
\end{proposition}

\begin{proof}
Assume $x_\alpha\fc 0$ in $E$. Then $\left|x_\alpha\right|u\oc 0$ in $E$ for all $u\in E_+$, and so $\left|x_\alpha\right|u\oc 0$ in $E^\delta$ for all $u\in E_+$; see Corollary 2.9 \cite{GTX}. Now, let's fix $v\in E^\delta_+$. Then there exists $x_v\in E_+$ such that $v\leq x_v$ because $E$ majorizes $E^\delta$. Then we have $\left|x_\alpha\right|v\leq \left|x_\alpha\right|x_v$. From $\left|x_\alpha\right|x_v\oc 0$ in $E^\delta$ it follows that $\left|x_\alpha\right|v\oc 0$ in $E^\delta$, that is, $x_\alpha\fc 0$ in the order completion $E^\delta$ because $v\in E^\delta_+$ is arbitrary.

Conversely, assume $x_\alpha\fc 0$ in $E^\delta$. Then, for all $u\in E^\delta_+$, we have $\left|x_\alpha\right|u\oc 0$ in $E^\delta$. In particular, for all $x\in E_+$, $\left|x_\alpha\right|x\oc 0$ in $E^\delta$. By Corollary 2.9 \cite{GTX}, we get $\left|x_\alpha\right|x\oc 0$ in $E$ for all $x\in E_+$. Hence $x_\alpha\fc$ in $E$. 
\end{proof}

The multiplication in $f$-algebra is $mo$-continuous in the following sense.
\begin{theorem}
Let $E$ be an infinite distributive $f$-algebra, and $(x_\alpha)_{\alpha \in A}$ and $(y_\beta)_{\beta \in B}$ be two nets in $E$. If $x_\alpha\fc x$ and $y_\beta\fc y$ for some $x,y\in E$ and each positive element of $E$ can be written as a multiplication of two positive elements then $x_\alpha y_\beta\fc xy$.
\end{theorem}

\begin{proof}
Assume $x_\alpha\fc x$ and $y_\beta\fc y$. Then $\left| x_\alpha-x\right|u\oc 0$ and $\left| y_\beta-y\right|u\oc 0$ for every $u\in E_+$. Let's fix $u\in E_+$. So, there exist another two nets $(z_\gamma)_{\gamma\in\Gamma}\downarrow 0$ and $(z_\xi)_{\xi\in\Xi}\downarrow 0$ in $E$ such that, for all $(\gamma,\xi)\in\Gamma\times\Xi$ there are $\alpha_\gamma\in A$ and $\beta_\xi\in B$ with $\left|x_\alpha-x\right|u\leq z_\gamma$ and $\left|y_\beta-y\right|u\leq z_\xi$ for all $\alpha\geq\alpha_\gamma$ and $\beta\geq\beta_\xi$.
	
Next, we show the $mo$-convergence of $(x_\alpha y_\beta)$ to $xy$. By considering the equality $\left|xy\right| =\left|x\right| \left|y\right|$ and Lemma \ref{basic monotonicity}, we have
\begin{eqnarray*}
\left|x_\alpha y_\beta-xy\right|u&=&\left|x_\alpha y_\beta-x_\alpha y+x_\alpha y-xy\right|u\\&\leq& \left|x_\alpha\right|\left|y_\beta-y\right|u+\left| x_\alpha -x\right|\left|y\right|u\\&\leq& \left|x_\alpha-x\right|\left|y_\beta-y\right|u+\left|x\right|\left|y_\beta-y\right|u+\left| x_\alpha -x\right|\left|y\right|u.
\end{eqnarray*}
The second and the third terms in the last inequality both order converge to zero as $\beta\to\infty$ and $\alpha\to \infty$ respectively because of $\left|x\right|u,\left|y\right|u\in E_+$, $x_\alpha\fc x$ and $y_\beta\fc y$. 
	
Now, let's show the convergence of the first term of last inequality. There are two positive elements $u_1,u_2\in E_+$ such that $u=u_1u_2$ because the positive element of $E$ can be written as a multiplication of two positive elements. So, we get $\left|x_\alpha-x\right|\left|y_\beta-y\right|u=(\left|x_\alpha-x\right|u_1)(\left|y_\beta-y\right|u_2)$. Since $(z_\gamma)_{\gamma\in\Gamma}\downarrow 0$ and $(z_\xi)_{\xi\in\Xi}\downarrow 0$, the multiplication $(z_\gamma z_\xi)\downarrow 0$. Indeed, we firstly show that the multiplication is decreasing. For indexes $(\gamma_1,\xi_1)(\gamma_2,\xi_2)\in \Gamma\times\Xi$, we have $z_{\gamma_2}\leq z_{\gamma_1}$ and $z_{\xi_2}\leq z_{\xi_1}$ because both of them are decreasing. Since the nets are positive, it follows from $z_{\xi_2}\leq z_{\xi_1}$ that $z_{\gamma_2}z_{\xi_2}\leq z_{\gamma_2}z_{\xi_1}\leq z_{\gamma_1}z_{\xi_1}$. As a result $(z_\gamma z_\xi)_{(\gamma,\xi)\in\Gamma\times\Xi}\downarrow$. Now, we show that the infimum of multiplication is zero. For a fixed index $\gamma_0$, we have $z_\gamma z_\xi\leq z_{\gamma_0}z_\xi$ for $\gamma\geq \gamma_0$ because $(z_\gamma)$ is decreasing. Thus, we get $\inf(z_\gamma z_\xi)=0$ because of $\inf(z_{\gamma_0}z_\xi)=z_{\gamma_0}\inf(z_\xi)=0$. Therefore, we see $(\left|x_\alpha-x\right|u_1)(\left|y_\beta-y\right|u_2)\oc 0$. Hence, we get $x_\alpha y_\beta\fc xy$.
\end{proof}

The lattice operations in an $f$-algebra are $mo$-continuous in the 
following sense.
\begin{proposition}\label{LO are $mo$-continuous}
Let $(x_\alpha)_{\alpha \in A}$ and $(y_\beta)_{\beta \in B}$ be two nets in an $f$-algebra $E$. If $x_\alpha\fc x$ and $y_\beta\fc y$ then $(x_\alpha\vee y_\beta)_{(\alpha,\beta)\in A\times B} \fc x\vee y$. In particular, $x_\alpha\fc x$ implies $x_\alpha^+\fc x^+$.
\end{proposition}

\begin{proof}
Assume $x_\alpha\fc x$ and $y_\beta\fc y$. Then there exist two nets $(z_\gamma)_{\gamma\in\Gamma}$ and $(w_\lambda)_{\lambda\in\Lambda}$ in $E$ satisfying $z_\gamma\downarrow 0$ and $w_\lambda\downarrow 0$, and for all $(\gamma,\lambda)\in\Gamma\times\Lambda$ there are $\alpha_\gamma\in A$ and $\beta_\lambda\in B$ such that $\left|x_\alpha-x\right|u\leq z_\gamma$ and $\left|y_\beta-y\right|u\leq w_\lambda$ for all $\alpha\geq\alpha_\gamma$ and $\beta\geq\beta_\lambda$ and for every $u\in E_+$. It follows from the inequality $|a\vee b-a\vee c|\leq |b-c|$ in vector lattices that 	
\begin{equation*}
\begin{split}
\left|x_\alpha \vee y_\beta - x\vee y\right|u&\leq \left|x_\alpha \vee y_\beta -x_\alpha \vee y\right|u+\left|x_\alpha \vee y- x\vee y\right|u\\ &\leq \left|y_\beta -y\right|u+\left|x_\alpha-x\right|u\leq w_\lambda+z_\gamma
\end{split}
\end{equation*}
for all $\alpha\geq\alpha_\gamma$ and $\beta\geq\beta_\lambda$ and for every $u\in E_+$. Since $(w_\lambda+z_\gamma)\downarrow 0$, $\left|x_\alpha \vee y_\beta - x\vee y\right|u$ order converges to $0$ for all $u\in E_+$. That is, $(x_\alpha\vee y_\beta)_{(\alpha,\beta)\in A\times B} \fc x\vee y$.
\end{proof}

\begin{lemma}\label{mo convergence positive}
Let $(x_\alpha)$ be a net in an $f$-algebra $E$. Then 
\begin{enumerate}
\item[(i)] $0\leq x_\alpha\fc x$ implies $x\in E_+$.
\item[(ii)] if $(x_\alpha)$ is monotone and $x_\alpha\fc x$ then implies $x_\alpha\oc x$.
\end{enumerate}
\end{lemma}

\begin{proof}
$(i)$ Assume $0\leq x_\alpha\fc x$. Then we have $x_\alpha=x_\alpha^+\fc x^+=0$ by Proposition \ref{LO are $mo$-continuous}. Hence, we get $x\in E_+$.
	
$(ii)$ We show that $x_\alpha\uparrow$ and $x_\alpha\fc x$ implies $x_\alpha\uparrow x$. Fix an index $\alpha$. Then we have $x_\beta-x_\alpha\in X_+$ for $\beta\ge\alpha$. By $(i)$, $x_\beta-x_\alpha\fc x-x_\alpha\in X_+$. Therefore, $x\geq x_\alpha$ for any $\alpha$. Since $\alpha$ is arbitrary, then $x$ is an upper bound of $(x_\alpha)$. Assume $y$ is another upper bound of $(x_\alpha)$, i.e., $y\geq x_\alpha$ for all $\alpha$. So, $y-x_\alpha\fc y-x\in X_+$, or $y\ge x$, and so $x_\alpha \uparrow x$.
\end{proof}

The following simple observation is useful in its own right.
\begin{proposition}
Every disjoint decreasing sequence in an $f$-algebra $mo$-converges to zero.
\end{proposition}

\begin{proof}
Suppose $(x_n)$ is a disjoint and decreasing sequence in an $f$-algebra $E$. So, $\left|x_n\right|u$ is also a disjoint sequence in $E$ for all $u\in E_+$; see Theorem 142.1(iii) \cite{Za}. Fix $u\in E_+$, by Corollary 3.6 \cite{GTX}, we have   $\left|x_n\right|u\uoc 0$ in $E$. So, $\left|x_n\right|u\wedge w\oc 0$ in $E$ for all $w\in E_+$. Thus, in particular for fixed $n_0$, taking $w$ as $\left|x_{n_0}\right|u$. Then, for all $n\geq n_0$, we get 
$$
\left|x_n\right|u=\left|x_n\right|u\wedge \left|x_{n_0}\right|u=\left|x_n\right|u\wedge w\oc 0
$$
because of $\left|x_n\right|u\leq \left|x_{n_0}\right|u$. Therefore, $x_n \fc 0$ in $E$.
\end{proof}

For the next two facts, observe the following fact. Let $E$ be a vector lattice, $I$ be an order ideal of $E$ and $(x_\alpha)$ be a net in $I$. If $x_\alpha\oc x$ in $I$ then $x_\alpha\oc x$ in $E$. Conversely, if $(x_\alpha)$ is order bounded in $I$ and $x_\alpha\oc x$ in $E$ then $x_\alpha\oc x$ in $I$.
\begin{proposition}
Let $E$ be an $f$-algebra, $B$ be a projection band of $E$ and $P_B$ be the corresponding band projection. If $x_\alpha\fc x$ in $E$ then $P_B(x_\alpha)\fc P_B(x)$ in both $E$ and $B$.
\end{proposition}

\begin{proof}
It is known that $P_B$ is a lattice homomorphism and $0\leq P_B\leq I$. It follows from $\lvert P_B(x_\alpha)-P_B(x)\rvert=P_B\lvert x_\alpha-x\rvert\leq \lvert x_\alpha-x\rvert$ that $\lvert P_B(x_\alpha)-P_B(x)\rvert u\leq \lvert x_\alpha-x\rvert u$ for all $u\in E_+$. Then it follows easily that $P_B(x_\alpha)\fc P_B(x)$ in both X and B.
\end{proof}

\begin{theorem}\label{ideal iff vector lattice}
Let $E$ be an $f$-algebra and $I$ be an order ideal and sub-$f$-algebra of $E$. For an order bounded net $(x_\alpha)$ in $I$, $x_\alpha\fc 0$ in $I$ iff $x_\alpha\fc 0$ in $E$.
\end{theorem}

\begin{proof}
Suppose $x_\alpha\fc 0$ in $E$. Then for any $u\in I_+$, we have $\lvert x_\alpha\rvert u\oc 0$ in $E$. So, the preceding remark implies $\lvert x_\alpha\rvert u\oc 0$ in $I$ because $\lvert x_\alpha\rvert u$ is order bounded in $I$ . Therefore, we get $x_\alpha\fc 0$ in $I$.
	
Conversely, assume that $(x_\alpha)$ $mo$-converges to zero in $I$. For any $u\in I_+$, we have $\lvert x_\alpha\rvert u\oc 0$ in $I$, and so in $E$. Then, by applying Theorem 142.1(iv) \cite{Za}, we have $x_\alpha w=0$ for all $w\in I^d=\{x\in E:x\perp y \ for \ all \ y\in I\}$ and for each $\alpha$ because $(x_\alpha)$ in $I$. For any $u\in I_+$ and each $0\leq w\in I^d$, it follows that 
$$
\lvert x_\alpha\rvert (u+w)=\lvert x_\alpha\rvert u+\lvert x_\alpha\rvert w=\lvert x_\alpha\rvert u\oc 0
$$
in $E$. So that, for each $z\in (I\oplus I^d)_+$, we get $\lvert x_\alpha\rvert z\oc 0$ in $E$. It is known that $I\oplus I^d$ is order dense in $E$; see Theorem 1.36 \cite{ABPO}. Fix $v\in E_+$. Then there exists some $u\in (I\oplus I^d)$ such that $v\leq u$. Thus, we have $\lvert x_\alpha\rvert v \leq \lvert x_\alpha\rvert u\oc 0$ in $E$. Therefore, $\lvert x_\alpha\rvert v\oc 0$, and so $x_\alpha\fc 0$ in $E$.
\end{proof} 

The following proposition extends Theorem 3.8 \cite{AAyd} to the general setting.
\begin{theorem}
Let $E$ be an infinite distributive $f$-algebra with a unit $e$ and $(x_n)\downarrow$ be a sequence in $E$. Then $x_n\fc 0$ iff $\left|x_n\right|(u\wedge e)\oc 0$ for all $u\in E_+$.
\end{theorem}

\begin{proof}
For the forward implication, assume  $x_n\fc 0$. Hence, $\left|x\right| u\oc 0$ for all $u\in E_+$, and so $\left|x_n\right|(u\wedge e)\leq \left|x_n\right|u\oc 0$ because of $e\in E_+$. Therefore, $\left|x_n\right|(u\wedge e)\oc 0$.

For the reverse implication, fix $u\in E_+$. By applying Theorem 2.57 \cite{ABPO} and Theorem 142.1(i) \cite{Za}, note that
\begin{eqnarray*}
\left| x_n\right| u\leq\left| x_n\right|(u-u\wedge ne)+\left| x_n\right|(u\wedge ne)\leq\frac{1}{n}u^2\left|x_n\right|+n\left|x_n\right|(u\wedge e)
\end{eqnarray*}
Since $(x_n)\downarrow$ and $E$ is Archimedean, we have $\frac{1}{n}u^2\left| x_n\right|\downarrow 0$. Furthermore, it follows from $\left|x_n\right|(u\wedge e)\oc 0$ for each $u\in E_+$ that there exists another sequence $(y_m)_{m\in B}$ satisfying $y_m \downarrow 0$, and for any $m\in B$, there exists $n_m$ such that $\left|x_n\right|(u\wedge e)\leq \frac{1}{n}y_m$, or $n\left|x_n\right|(u\wedge e)\leq y_m$ for all $n \geq n_m$. Hence, we get $n\left|x_n\right|(u\wedge e)\oc 0$. Therefore, we have $\left|x_n\right|u\oc 0$, and so $x_n\fc 0$.
\end{proof}

The $mo$-convergence passes obviously to any sub-$f$-algebra $Y$ of $E$, i.e., for any net $(y_\alpha)$ in $Y$, $y_\alpha\fc0$ in $E$ implies $y_\alpha\fc0$ in $Y$. For the converse, we give the following theorem.
\begin{theorem}\label{$up$-regular}
Let $Y$ be a sub-$f$-algebra of an $f$-algebra $E$ and $(y_\alpha)$ be a net in $Y$. If $y_\alpha\fc 0$ in $Y$ then it $mo$-converges to zero in $E$ for each of the following cases;
\begin{enumerate}
\item[(i)] $Y$ is majorizing in $E$;
\item[(ii)] $Y$ is a projection band in $E$;
\item[(iii)] if, for each $u\in E$, there are element $x,y\in Y$ such that $\left|u-y\right|\leq |x|$.
\end{enumerate}	
\end{theorem}

\begin{proof}
Assume $(y_\alpha)$ is a net in $Y$ and $y_\alpha\fc 0$ in $Y$. Let's fix $u\in E_+$.
	
$(i)$ Since $Y$ is majorizing in $E$, there exists $v\in Y_+$ such that $u\leq v$. It follows from
$$
0\leq |y_\alpha|u\leq |y_\alpha| v\oc0,
$$
that $|y_\alpha|u\oc 0$ in $E$. That is, $y_\alpha\fc 0$ in $E$.
	
$(ii)$ Since $Y$ is a projection band in $E$, we have $Y=Y^{\bot\bot}$ and $E=Y\oplus Y^{\bot}$. Hence $u=u_1+u_2$ with $u_1\in Y_+$ and $u_2\in Y^{\bot}_+$. Thus, we have $y_\alpha\wedge u_2=0$ because $(y_\alpha)$ in $Y$ and $u_2\in Y^{\bot}$. Hence, by applying Theorem 142.1(iii) \cite{Za}, we see $y_\alpha u=0$ for all index $\alpha$. It follows from
$$
\left|y_\alpha\right|u=\left|y_\alpha\right|(u_1+u_2)=\left|y_\alpha\right|u_1\oc 0
$$ 
tha $\left|y_\alpha\right|u\oc 0$ in $E$. Therefore, $y_\alpha\fc 0$ in $E$.

$(iii)$ For the given $u\in E_+$, there exists elements $x,y\in Y$ with $\left|u-y\right|\leq |x|$. Then 
$$
\left|y_\alpha\right|u\leq \left|y_\alpha\right|\left|u-y\right|+\left|y_\alpha\right|\left|y\right|\leq\left|y_\alpha\right|\left|x\right|+\left|y_\alpha\right|\left|y\right|.
$$
By $mo$-convergence of $(y_\alpha)$ in $Y$, we have $\left|y_\alpha\right|\left|x\right|\oc 0$ and $\left|y_\alpha\right|\left|y\right|\oc 0$, and so $\left|y_\alpha\right|u\oc 0$. That means $y_\alpha\fc 0$ in $E$ because $u$ is arbitrary in $E_+$.
\end{proof}

We continue with some basic notions in $f$-algebra, which are motivated by their analogies from vector lattice theory.
\begin{definition}\label{$mo$-notions}
Let $(x_\alpha)_{\alpha \in A}$ be a net in $f$-algebra $E$. Then 
\begin{enumerate}
\item[(i)]  $(x_\alpha)$ is said to be {\em $mo$-Cauchy} if the net $(x_\alpha-x_{\alpha'})_{(\alpha,\alpha') \in A\times A}$ $mo$-converges to $0$,
\item[(ii)] $E$ is called {\em $mo$-complete} if every $mo$-Cauchy net in $E$ is $mo$-convergent,
\item[(iii)] $E$ is called {\em $mo$-continuous} if $x_\alpha\oc0$ implies $x_\alpha\fc 0$,
\item[(iv)] $E$ is called a {\em $mo$-KB-space} if every order bounded increasing net in $E_+$ is $mo$-convergent.
\end{enumerate}	
\end{definition}

\begin{remark}\label{$mo$-cont-0}
An $f$-algebra $E$ is $mo$-continuous iff $x_\alpha\downarrow 0$ in $E$ implies $x_\alpha\fc 0$. Indeed, the implication is obvious. For the converse, consider a net $x_\alpha\oc 0$. Then there exists a net $z_\beta\downarrow 0$ in $X$ such that, for any $\beta$ there exists $\alpha_\beta$ so that $|x_\alpha|\leq z_\beta$ for all $\alpha\geq\alpha_\beta$. Hence, by $mo$-continuity of $E$, we have $z_\beta\fc 0$, and so $x_\alpha\fc 0$.
\end{remark}

\begin{proposition}
Let $(x_\alpha)$ be a net in an $f$-algebra $E$. If $x_\alpha\fc x$ and $(x_\alpha)$ is an $o$-Cauchy net then $x_\alpha\oc x$. Moreover, if $x_\alpha\fc x$ and $(x_\alpha)$ is $uo$-Cauchy then $x_\alpha\uoc x$.
\end{proposition}

\begin{proof}
Assume $x_\alpha\fc x$ and $(x_\alpha)$ is an order Cauchy net in $E$. Then $x_\alpha-x_\beta\oc 0$ as $\alpha,\beta\to\infty$. 
Thus, there exists another net $z_\gamma\downarrow 0$ in $E$ such that, for every $\gamma$, there exists $\alpha_\gamma$ satisfying 
$$
|x_\alpha-x_\beta|\leq z_\gamma
$$
for all $\alpha,\beta\geq \alpha_\gamma$. By taking $f$-limit over $\beta$ the above inequality and applying Proposition \ref{LO are $mo$-continuous}, i.e., $\left|x_\alpha-x_\beta\right|\fc \left|x_\alpha-x\right|$, we get $|x_\alpha-x|\leq z_\gamma$ for all $\alpha\geq\alpha_\gamma$. That means $x_\alpha\oc x$. The similar argument can be applied for the $uo$-convergence case, and so the proof is omitted.
\end{proof}

In the case of $mo$-complete in $f$-algebras, we have conditions for $mo$-continuity.
\begin{theorem}\label{of-contchar}
For an $mo$-complete $f$-algebra $E$, the following statements are equivalent:
\begin{enumerate}
\item[(i)] $E$ is $mo$-continuous;
\item[(ii)] if $0\leq x_\alpha\uparrow\leq x$ holds in $E$ then $x_\alpha$ is a $mo$-Cauchy net;
\item[(iii)] $x_\alpha\downarrow 0$ implies $x_\alpha\fc 0$ in $E$.
\end{enumerate}	
\end{theorem}

\begin{proof} 
$(i)\Rightarrow(ii)$ Consider the increasing and bounded net $0\leq x_\alpha\uparrow\leq x$ in $E$. Then there exists a net $(y_\beta)$ in $E$ such that $(y_\beta-x_\alpha)_{\alpha,\beta}\downarrow 0$; see Lemma 12.8 \cite{ABPO}. Thus, by applying Remark \ref{$mo$-cont-0}, we have $(y_\beta-x_\alpha)_{\alpha,\beta}\fc0$, and so the net $(x_\alpha)$ is $mo$-Cauchy because of $\left|x_\alpha-x_{\alpha'} \right|_{\alpha,\alpha'\in A}\leq\left|x_\alpha-y_\beta\right|+\left|y_\beta-x_{\alpha'}\right|$.
	
$(ii)\Rightarrow(iii)$ Suppose that $x_\alpha\downarrow 0$ in $E$, and fix arbitrary $\alpha_0$. Then we have $x_\alpha\leq x_{\alpha_0}$ for all $\alpha\geq\alpha_0$. Thus we can get $0\leq(x_{\alpha_0}-x_\alpha)_{\alpha\geq\alpha_0}\uparrow\leq x_{\alpha_0}$. So, it follows from $(ii)$ that the net $(x_{\alpha_0}-x_\alpha)_{\alpha\geq\alpha_0}$ is $mo$-Cauchy, i.e., $(x_{\alpha^{'}}-x_\alpha)\fc 0$ as $\alpha_0\le\alpha,\alpha^{'}\to\infty$. Then there exists $x\in E$ satisfying $x_\alpha\fc x$ as $\alpha_0\le\alpha\to\infty$ because $E$ is $mo$-complete. Since $x_\alpha\downarrow$ and $x_\alpha\fc 0$, it follows from Lemma \ref{mo convergence positive} that $x_\alpha\downarrow0$, and so we have $x=0$. Therefore, we get $x_\alpha\fc 0$. 
	
$(iii)\Rightarrow(i)$ It is just the implication of Remark \ref{$mo$-cont-0}.
\end{proof}

\begin{corollary}\label{of + f implies o}
Let $E$ be an $mo$-continuous and $mo$-complete $f$-algebra. Then $E$ is order complete. 
\end{corollary}

\begin{proof}
Suppose $0\leq x_\alpha\uparrow\leq u$ in $E$. We show the existence of supremum of $(x_\alpha)$. By considering Theorem \ref{of-contchar} $(ii)$, we see that $(x_\alpha)$ is an $mo$-Cauchy net. Hence, there is $x\in E$ such that $x_\alpha\fc x$ because $E$ is $mo$-complete. It follows from Lemma \ref{mo convergence positive} that $x_\alpha\uparrow x$ because of $x_\alpha\uparrow$ and $x_\alpha\fc x$. Therefore, $E$ is order complete.
\end{proof}

\begin{proposition}\label{f-KB is of}
Every $mo$-KB-space is $mo$-continuous. 
\end{proposition}

\begin{proof} 
Assume $x_\alpha\downarrow 0$ in $E$. From Theorem \ref{of-contchar}, it is enough to show $x_\alpha\fc 0$. Let's fix an index $\alpha_0$, and define another net  $y_\alpha:=x_{\alpha_0}-x_\alpha$ for $\alpha\ge\alpha_0$. Then it is clear that $0\leq y_\alpha\uparrow\le x_{\alpha_0}$, i.e., $(y_\alpha)$ is increasing and order bounded net in $E$. Since $E$ is a $mo$-KB-space, there exists $y\in E$ such that $y_\alpha\fc y$. Thus, by Lemma \ref{mo convergence positive}, we have $y_\alpha\oc y$. Hence, $y=\sup\limits_{\alpha\geq\alpha_0}y_\alpha=\sup\limits_{\alpha\geq\alpha_0}(x_{\alpha_0}-x_\alpha)=x_{\alpha_0}$ because of $x_\alpha\downarrow 0$. Therefore, we get $y_\alpha=x_{\alpha_0}-x_\alpha\fc x_{\alpha_0}$ or $x_\alpha\fc0$ because of $y_\alpha\fc y$. 
\end{proof}

\begin{proposition}
Every $mo$-KB-space is order complete.
\end{proposition}

\begin{proof}
Suppose $0\leq x_\alpha\uparrow\leq z$ is an order bounded and increasing net in an $mo$-KB-space $E$ for some $z\in E_+$. Then $x_\alpha\fc x$ for some $x\in E$ because $E$ is $mo$-KB-space. By Lemma \ref{mo convergence positive}, we have $x_\alpha\uparrow x$ because of $x_\alpha\uparrow$ and $x_\alpha\fc x$. So, $E$ is order complete.
\end{proof}

\begin{proposition}\label{$o$-closed sublattice of $mo$-KB}
Let $Y$ be an sub-$f$-algebra and order closed sublattice of an $mo$-KB-space $E$. Then $Y$ is also a $mo$-KB-space.
\end{proposition}

\begin{proof}
Let $(y_\alpha)$ be a net in $Y$ such that $0\leq y_\alpha\uparrow\leq y$ for some $y\in Y_+$. Since $E$ is a $mo$-KB-space, there exists $x\in E_+$ such that $y_\alpha\fc x$. By Lemma \ref{mo convergence positive}, we have $y_\alpha\uparrow x$, and so $x\in Y$ because $Y$ is order closed. Thus $Y$ is a $mo$-KB-space.
\end{proof}

\end{document}